%% file: main.tex
\documentclass[11pt]{article}
\usepackage{mathtext} 			
\usepackage[T2A]{fontenc}			
\usepackage[utf8]{inputenc}			
\usepackage[english,russian]{babel}	
\usepackage{mathrsfs,stackengine}
\usepackage{amsmath,amsfonts,amssymb,amsthm,mathtools} 
\usepackage{icomma} 
\usepackage{graphicx}
\graphicspath{{images/}}
\usepackage{authblk}

\input{macros}

\title{Унификация в подсистеме $\j$  логики доказуемости $\mathbf{GLB}$}
\author{Н.В. Лукашов \footnote{Работа поддержана программой <<Научный фонд НИУ ВШЭ>>, грант 23-00-022.}\\ lnv619@gmail.com}
\affil{факультет математики\\ НИУ <<Высшая школа экономики>>}

\date{июнь 2023 г.}
\renewcommand{\phi}{\varphi}

\usepackage{enumerate}
\usepackage{enumitem}
\usepackage{ytableau}

\renewcommand{\l}{\langle}
\renewcommand{\r}{\rangle}
\newcommand{\w}{\mathcal{W}}
\renewcommand{\j}{\mathbf{J}_2}
\newcommand{\with}{\, \& \,}
\renewcommand{\t}{\overline{\theta}}

\renewcommand{\u}{\mathcal{U}}
\renewcommand{\a}{\mathcal{A}}
\usepackage{mathtools}
\usepackage{subcaption}
\usepackage[e]{esvect}

\begin{document}

    \maketitle
   \begin{abstract}
	Мы обобщим методы С. Гилярди, разработанные им в \cite{ghilardi2000best}, и применим их к  подсистемы $\mathbf{J}_2$ бимодальной логики доказуемости \textbf{GLB}. Мы опишем проективные формулы в $\mathbf{J}_2$ в терминах семантики Крипке и с помощью него докажем, что  логика $\j$ имеет конечный тип унификации. В качестве применения полученных результатов, мы  покажем разрешимость проблемы допустимости правил вывода логики  $\j$.
\end{abstract}
	\section{Введение}
	
	Бимодальная логика доказуемости \textbf{GLB} была введена Г.K. Джапаридзе~\cite{japaridze} в 1985 году как расширение логики доказуемости Гёделя-Лёба \textbf{GL} с модальностью $\square$ (она же $[0]$), интерпретируемую как доказуемость в арифметике Пеано,  добавлением новой модальности $[1]$, интерпретируемую как $\omega$-доказуемость в PA.  
	
	Основная трудность в изучении логики \textbf{GLB}, как было показано самим Г.К. Джапаридзе, заключается в отсутствии полноты относительно любого класса шкал Крипке: если $\l W, R_0, R_1\r$ --- шкала Крипке для $\mathbf{GLB}$, то с неизбежностью $R_1 = \varnothing$. Тем не менее, Г.К. Джапаридзе \cite{japaridze} смог показать арифметическую полноту и разрешимость \textbf{GLB}.
	
	После К. Игнатьев \cite{ignatiev1992closed, ignatiev1993strong} установил для \textbf{GLB} интерполяционное свойство Крейга, теорему о неподвижной точке, теорему о нормальной форме замкнутых формул.
	
	Важные результаты об унификационном типе логики \textbf{GL} были получены С. Гилярди \cite{ghilardi2000best} в 2000 г. Он показал, что наличие проективного унификатора для формулы $\phi$ (т.е. такой подстановки $\sigma$, что $\vdash_\mathbf{GL} \sigma(\phi)$ и $\phi \vdash_\mathbf{GL} \sigma(p) \leftrightarrow p$ для любой переменной $p$) эквивалентно  наличию для  её класса  моделей Крипке $MOD_\mathbf{GL}(\phi)$ некоторого \textit{свойства расширения}. Дальше С. Гилярди установил, что любая унифицируемая формула в \textbf{GL} имеет конечный базис унификаторов (т.е  логика \textbf{GL} имеет \textit{конечный} тип унификации), и  в терминах проективной аппроксимации  дал описание допустимых правил в логике~\textbf{GL}. 
 
	Первые шаги в решении аналогичных проблем для логики \textbf{GLB} были сделаны  Д.~Макаровым в его выпускной квалификационной работе~\cite{makarov}, однако работа не была доведена до конца. 
	
	В данной работе мы обощим методы С. Гилярди и получим аналогичные результаты для подсистемы $\mathbf{J}_2$ логики \textbf{GLB}, введённой Л.Д.~Беклемишевым, которая  уже полна по Крипке относительно так называемых \textit{стратифицированных} моделей. Мы  получим описание проективных формул в логике $\mathbf{J}_2$ в терминах семантики Крипке и покажем финитный тип унификации $\mathbf{J}_2$. В~заключение, мы опишем допустимые правила логики $\mathbf{J}_2$, используя   проективную аппроксимацию, и покажем, что проблема допустимости правил вывода в $\j$ алгоритмически разрешима.

 Отметим, что вопросы о типе унификации для логики \textbf{GLB} и разрешимости проблемы допустимости правил вывода для \textbf{GLB} пока остаются открытыми.

	\section{Предварительные сведения}

	\subsection{Базовые понятия.}
	\textit{Язык} бимодальной пропозициональный логики состоит из  пропозициональных переменных $p_1, p_2, \ldots$, констант $\top$ и $\bot$, булевых связок $\wedge, \vee, \neg, \to$ и модальностей $[0]$ и $[1]$. При этом модальность $\l i\r$ понимается как сокращение $\neg[i]\neg$. \textit{Модальная глубина}  $d(\phi)$ формулы $\phi$ определяется индукцией по построению: $d(p_i) = 0, d(\bot)=d(\top)=~0, \ d(\phi \circ \psi)=max\{d(\phi), d(\psi)\}$ для булевых связок $\circ$, $d([i]\phi) = 1+d(\phi)$ для каждого~$i$.
	
	\textit{Шкалой Крипке}  $\l W, R_0, R_1\r$ для языка бимодальной логики называется непустое множество $W$ (множество \textit{миров}) вместе с бинарными отношениями $R_0$ и $R_1$ на $W$ (отношения \textit{достижимости}).  
	
	\textit{Моделью Крипке} $\l W, R_0, R_1, v\r$ называется шкала Крипке $\l W, R_0, R_1\r$ вместе с \textit{оценкой переменных} $v$ --- функцией, сопоставляющей каждой пропозициональной переменной подмножество $W$ (\textit{множество истинности} $p_i$). Примем соглашение, что  $x \in \w$ означает $x\in W$. 
	
	Для модели Крипке $\w$ и мира $x \in \w$ мы можем рассмотреть \textit{пораждённую подмодель} $\w_x$, определяемую как наименьшее подмножество носителя $W$, которое содержит мир  $x$ и такое, что $y\in \w_x \with yR_iz \Rightarrow z \in \w_x$ для каждого $i$. При этом  $x$ называется \textit{корнем} модели $\w_x$.

	 Индукцией по построению формулы $\phi$ определим  её истинность в модели $\w$ в мире $x\in \w$ (обозначение $\w, x \Vdash \phi$):
	 \begin{itemize}
	 	\item $\w, x \Vdash p_i \Leftrightarrow x\in v(p_i)$;
	 		\item $\w, x \Vdash \top$, $\w, x \not\Vdash \bot$;
	 			\item $\w, x \Vdash \phi_1 \land \phi_2 \Leftrightarrow (\w, x \Vdash \phi_1 \text{ и } \w, x \Vdash \phi_2)$;
	 				\item $\w, x \Vdash \phi_1 \vee \phi_2 \Leftrightarrow( \w, x \Vdash \phi_1 \text{ или } \w, x \Vdash \phi_2)$;
	 					\item $\w, x \Vdash \phi_1 \to \phi_2 \Leftrightarrow (\w, x \Vdash \phi_1  \Rightarrow \w, x \Vdash \phi_2)$;
	 						\item $\w, x \Vdash \neg\phi_1  \Leftrightarrow \w, x \not\Vdash \phi_1 $;
	 							\item $\w, x \Vdash [i]\phi_1  \Leftrightarrow \forall y \ (xR_iy \Rightarrow \w, y \Vdash \phi_1)$.
	 \end{itemize}
 
 Если формула $\phi$ истинна во всех мирах модели $\w$, то будем писать $\w \vDash \phi$ и говорить, что  $\phi$ \textit{глобально истинна} в модели $\w$.  Дальше в работе через $L$ будем обозначат логику \textbf{GLB} или её подсистемы. Если логика $L$ обладает свойством конечных моделей,  то  $MOD_L(\phi)$  --- множество конечных	моделей с корнем, в  каждом мире которых истинна  формула~$\phi$.

	\subsection{Логика GLB и её подсистемы} 
	Пропозициональная бимодальная логика  $\mathbf{GLB}$    c двумя модальностями $[0]$ и $[1]$ задаётся следующими схемами аксиом и правилами вывода: 
	\begin{enumerate}[leftmargin=*, align=left]
		\item[\textit{Аксиомы:}] 	\begin{enumerate}[label=(\roman*)]
			\item \label{glb_1}все булевы тавтологии;
			\item $[i](\phi \to \psi) \to \left([i] \phi \to \left([i] \psi\right)\right)$, $i = 0,1$;
			\item $[i]([i]\phi \to  \phi) \to [i]\phi$, $i=0, 1$;
			\item $[m]\phi \to [n][m] \phi$, для $m\leqslant n$;
			\item\label{glb_5} $\langle0\rangle \phi \to [1] \langle 0 \rangle \phi$;
			\item \label{glb_4} $[0] \phi\to [1]\phi$.
		\end{enumerate} 
	
		\item[\textit{Правила вывода:}] modus ponens; $\phi \vdash [i]\phi$, $i=0, 1$.
	\end{enumerate}

	Для преодоления трудностей, связанных с отсутствием полноты у \textbf{GLB} относительно какого-нибудь класса шкал Крипке, были предложены подсистемы  $\mathbf{GLB}$, которые уже являются полными по Крипке.
	
	К. Игнатьев впервые  выделил отдельно аксиомы \ref{glb_1}-\ref{glb_5} и рассмотрел соответствующую подсистему, которую мы будем обозначать через  \textbf{I}. Игнатьев показал, что логика \textbf{I} полна относительно класса шкал  Крипке $\l W, R_0, R_1\r$, удовлетворяющим двум условиям (будем называть        такие шкалы \textit{шкалами Игнатьева}):
	 \begin{itemize}
	 	\item $R_i$ обратно фундированное, иррефлексивное, транзитивное отношение на $W$, для каждого $i=0, 1$;
	 	\item $
	 	\forall x, y \ (xR_1y \Rightarrow \forall z \ (xR_0z\Leftrightarrow yR_0z))$ --- рис. \ref{fig_ignat}. 
	 \end{itemize}
 
		\begin{figure}[t]
			\centering
			\begin{minipage}{0.49\textwidth}
				\centering
				\includegraphics[height=0.41\textwidth]{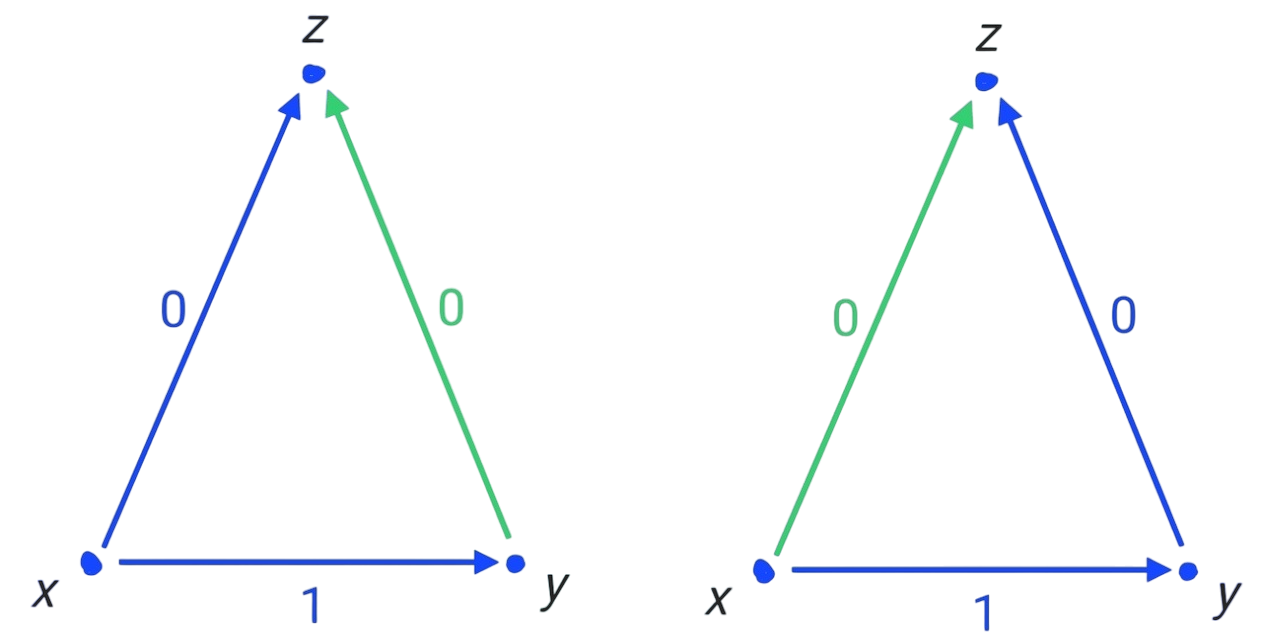}
				\caption{\textit{шкалы Игнатьева}}
				\label{fig_ignat}
			\end{minipage}
						\begin{minipage}{0.24\textwidth}
												\centering
							\includegraphics[height=0.85\textwidth]{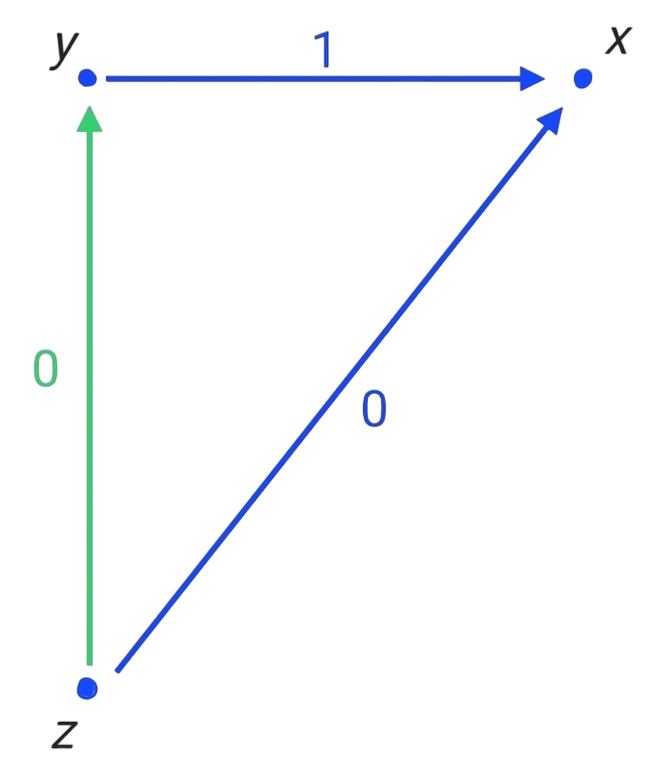}
							\caption{$\mathbf{J}_2$-\textit{шкалы}}
							\label{fig_jframes}
						\end{minipage}
		\begin{minipage}{0.24\textwidth}
		\centering
		\includegraphics[height=0.85\textwidth]{str3.png}
		\caption{ (\ref{condition_stratif})}
		\label{fig_stratif}
	\end{minipage}
 			\end{figure}

	После Л.Д. Беклемишевым была рассмотрена подсистема $\mathbf{J}_2$, получающаяся из \textbf{I} добавлением схемы аксиом $[m]\phi \to [m][n]\phi$ для $m  \leqslant n$, которые выводимы в \textbf{GLB}. $\mathbf{J}_2$-\textit{шкалой} называется шкала Игнатьева, удовлетворяющая условию: 
	\begin{itemize}
		\item $\forall x, y \ (xR_my \with yR_nx \Rightarrow xR_m z),$ если  $m\leqslant n$ --- рис. \ref{fig_jframes}.
	\end{itemize}
	
	В результате, Л.Д. Беклемишев установил соответствующую полноту:\footnote{\cite{beklemishev2010kripke}, теорема 1.}
	\begin{theorem}
		Логика $\mathbf{J}_2$ корректна и полна относительно (конечных) $\mathbf{J}_2$-шкал.
	\end{theorem}
	    
		Положим $E_m$ --- симметричное, транзитивное, рефлексивное замыкание $R_m$. Классы эквивалентности $E_m$ называются \textit{$m$-листами} или \textit{$m$-слоями.} Непосредственно из определения, имеем следующие свойства $m$-листов: 
		\begin{itemize}
			\item Любой $0$-лист разбивается на $1$-листы.
			\item Все точки $1$-листа не сравнимы между собой отношению  $R_0$ (иначе, из свойств логики $\j$, был бы рефлексивный мир).

			\item Существует отношение упорядочивания $R_0$ на $1$-листах, определяемое как 
			$$\alpha R_0 \beta ,\text{\quad если } \exists x \in \alpha  \ \exists y \in \beta \ xR_0y.$$
			Более того, так как $xR_1y \Rightarrow \forall z \ (xR_0z {\Leftrightarrow }yR_0z)$, то 
			$$\alpha R_0 \beta \Longleftrightarrow \exists y \in \beta \ \forall x \in \alpha  \ xR_0y. $$
		\end{itemize}

	\subsection{Стратифицируемость}
	Оказывается, что можно рассмотреть ещё более специализированный класс моделей Крипке логики $\j$, относительно которого по-прежнему будет полнота. 
	
	\begin{definition}
	Шкала логики $\mathbf{J}_2$ называется \textit{стратифицированной}, если для неё выполнено следующие дополнительное условие (рис. \ref{fig_stratif}):
		\begin{flalign}\label{condition_stratif}
		\forall x, y, z  \ (zR_m x \with y R_n x \Rightarrow zR_my), \quad \text{если } m<n. \tag{S}
		\end{flalign}
	\end{definition}
	
	Тогда в стратифицированных шкалах для любых $1$-листов $\alpha$ и $\beta$, таких что $\alpha R_0\beta$, каждая точка листа $\beta$ $R_0$-достижима из любой точки листа $\alpha$. Поэтому  $R_0$-упорядочивание в стратифицированных шкалах полностью задаётся их  $R_0$-упорядочиванием 1-листов. Таким образом,  стратифицированные модели можно представлять себе как трёхмерные структуры (см. рис.  \ref{fig_str_models}). Следующая теорема доказана Л.Д. Беклемишевым:\footnote{\cite{beklemishev2010kripke}, теорема 2.}
	
	\begin{theorem} \label{stratification}
				Логика $\mathbf{J}_2$ корректна и полна относительно (конечных) стратифицированных шкал.
	\end{theorem}

		\begin{figure}[t]
			\centering 
			\includegraphics[height=0.25\textheight]{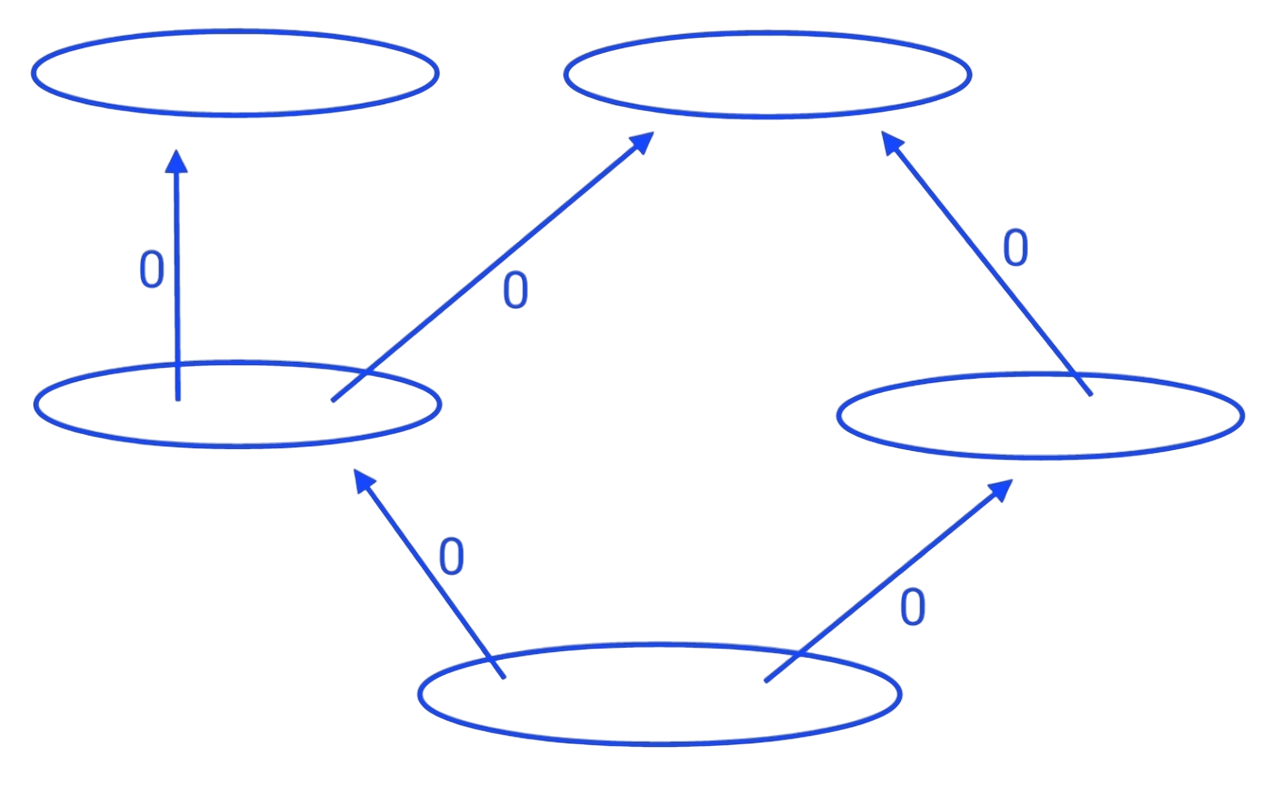}
			\caption{ \textit{стратифицированные модели}}
			\label{fig_str_models}
		\end{figure}

	\subsection{Бисимуляции}  Отношение эквивалентности $\sim_n$, называемое \textit{$n$-бисимуляцией}, между двумя моделями  определяется индукцией: 
	\begin{itemize}
		\item $\mathcal{W}_x \sim_0 \mathcal{W}'_{x'}$, если в $x$ и $x'$ истинны одни и те же пропозициональные переменные ($x \Vdash p \Leftrightarrow x' \Vdash p$). 
		\item $\mathcal{W}_x \sim_{n+1}\mathcal{W}'_{x'}$, если 
			\begin{enumerate}[label=(\roman*)]
			\item $\mathcal{W}_x \sim_0 \mathcal{W}'_{x'}$;
			\item $\forall y \in \w_x\ (xR_iy \Rightarrow \exists y' \ (x'R_iy' \with \w_y\sim_n\w'_{y'}))$ для любого $i=0, 1$;
			\item $\forall y' \in \w'_{x'}\ (x'R_iy' \Rightarrow \exists y \ (xR_iy \with \w'_{y'}\sim_n\w_y))$ для любого $i=0, 1$.
			\end{enumerate}
	\end{itemize}
	
	Из определения нетрудно видеть, что если $m>n$  и $\mathcal{W}_x \sim_m \mathcal{W}_{x'}$, то $\mathcal{W}_x \sim_n \mathcal{W}_{x'}$. Также $n$-бисимуляция --- это отношение эквивалентности. Класс эквивалетности модели $\w$ по этому отношению будем обозначать $[\w]_n$. Индукцией, нетрудно убедиться, что для каждого фиксированного $n$ количество классов эквивалентности $[\w]_n$ конечно.
	
	Для конечных моделей будем писать $\mathcal{W}_x \sim_\infty \mathcal{W}_{x'}$, если $\forall n \in \mathbb{N} \ \mathcal{W}_x \sim_n \mathcal{W}_{x'}$. Истинные смысл  и пользу $n$-бисимуляции показывает следующее предложение:
	\begin{proposition} \label{bissimulation}
		$\mathcal{W}_x \sim_n \mathcal{W'}_{x'}$ тогда и только тогда, когда для любой формулы $\phi$, такой что $d(\phi) \leqslant n$, выполнено $(\mathcal{W}_x, x\Vdash \phi \Leftrightarrow \mathcal{W'}_{x'}, x' \Vdash \phi)$.
	\end{proposition}

	\begin{proof}
		$(\Rightarrow)$ Доказывается непосредственно индукцией по $n$.
		
		$(\Leftarrow)$ Для простоты обозначений положим $\mathcal{W} := \mathcal{W}_x$ и $\mathcal{W}':=\mathcal{W}'_{x'}$. Достаточно показать, что существуют такая формула $X^n_{\mathcal{W}}$ глубины не более $n$, что 
		\begin{equation}\label{theorem_2}
			\mathcal{W}', x' \Vdash X^n_{\mathcal{W}} \ \Longleftrightarrow \ \mathcal{W} \sim_n \mathcal{W}'
		\end{equation}
		для любой модели $\mathcal{W}'$. Сделаем это следующим образом: для $n=0$ возьмём  $X^0_\mathcal{W} {=} \bigwedge_{x \Vdash p_i} p_i \land \bigwedge_{x \not\Vdash p_i} \neg p_i$ и для $n>0$ положим: 
		$$X^n_\mathcal{W} = \bigwedge_{x \Vdash p_i} p_i \land 
		\bigwedge_{x \not\Vdash p_i} \neg p_i \land 
		\bigwedge_{i}\bigwedge_{xR_iy} \l i\r X^{n-1}_{\mathcal{W}_y}\land
		\bigwedge_{i} [i]\left(\bigvee_{xR_iy} X^{n-1}_{\mathcal{W}_y}\right).$$
		
		Докажем истинность утверждения (\ref{theorem_2}) индукцией по $n$. База следует непосредственно из построения. Предположим, что  (\ref{theorem_2}) верно для $n-1$ и  $\mathcal{W} \sim_n \mathcal{W}'$. Тогда, по определению $(n-1)$-бисимуляции, в мире $x'$ модели $\w'$ верен каждый конъюнктивный член. Значит $\mathcal{W}', x' \Vdash X^n_{\mathcal{W}}$.

		Обратно, пусть $\mathcal{W}', x' \Vdash X^n_{\mathcal{W}}$,  и проверим, что выполнены условия $n$-бисимуляции. Рассмотрим мир $y\in \w$, такой что $xR_iy$ для некоторого $i$. Тогда $\w', x' \Vdash  \l i\r X^{n-1}_{\w_y}$, значит найдётся $y'\in \w'$ со свойством $\w', y' \Vdash X^{n-1}_{\w_y}$, что эквивалентно $\w_y \sim_{n-1}\w'_{y'}$, по предположению. Рассмотрим теперь $y'\in \w'$, такой что $x'R_iy'$. Тогда, в $y'$ верна дизъюнкция $X^{n-1}_{\mathcal{W}_y}$ по всем $y$, $R_i$-достижимым из $x\in \w$, а значит $\w', y' \Vdash X^{n-1}_{\mathcal{W}_z}$ для некоторого $z\in \w$. Отсюда $\w'_{y'}\sim_{n-1}\w_y $. 
		
	\end{proof}

	\subsection{Подстановки} Зафиксируем раз и навсегда конечное множество переменных  $\vv{p} =(p_1, p_2, \ldots,  p_n)$. Формулу от этих переменных будем обозначать соотвественно $\phi(\vv{p})$. Положим  $Form(\vv{p})$ --- множество всех формул в языке бимодальной логики вида $\phi(\vv{p})$. 
	
	\textit{Подстановкой} $\sigma$  называется функция $\sigma: \vv{p} \to Form(\vv{p})$, сопоставляющая каждой переменной из списка некоторую формулу.  
	
	Определим $\sigma(\phi(\vv{p})) \leftrightharpoons \phi(p_1/\sigma(p_1), \ldots, p_n / \sigma(p_n))$. Таким образом, $\sigma$ может быть расширена на область определения $Form(\vv{p})$.  
	
	\textit{Композиция} подстановок $\tau$ и $\sigma$ определяется как $(\tau\sigma)(p)= \tau(\sigma(p))$ для всех $p \in \vv{p}$. Подстановка $\sigma_1$ \textit{менее общая}, чем $\sigma_2$ (обозначение \linebreak$\sigma_1~\leqslant \sigma _2$), если найдётся такая подстановка $\tau$, что для всех $p\in \vv{p}$ 
	$$\vdash_L \tau(\sigma_2(p)) \leftrightarrow \sigma_1(p).$$
	
	Подстановка $\sigma$ называется \textit{унификатором} для формулы $\phi(\vv{p})$ в логике $L$, если $\vdash_L \sigma(\phi)$. Унификатор $\sigma_1$  \textit{менее общий }унификатора $\sigma_2$, если он меннее общий как подстановка. 
	
	Множество $S$ унификаторов для $\phi$  называется \textit{полным}, если любой унификатор для $\phi$ менее общий для какого-нибудь унификатора из $S$. Полное множество унификаторов  $S$ для $\phi$  называется \textit{базисом}, если любые два элемента из $S$  не сравнимы относительно предпорядка $\leqslant$. Унификатор $\sigma$ для $\phi$ называется \textit{самым общим}, если $\{\sigma\}$ является полным множеством унификаторов.
	
	Для данной подстановки $\sigma$ можно сопоставить модели Крипке $\mathcal{W}=\l W, R_0, R_1, v\r$ новую модель Крипке $\sigma(\mathcal{W})  = \l W, R_0, R_1, \sigma(v)\r$, положив 
	$$\sigma(\mathcal{W}), x \Vdash p_i \ \overset{def }{\Longleftrightarrow}\ \mathcal{W}, x \Vdash \sigma(p_i)$$
	для каждого мира $x$ и каждой переменной $p_i$.
	
	Заметим, что подстановка $\sigma$, применённая к моделям, коммутирует с ограничением на миры: $\sigma(\mathcal{W})_x = \sigma(\mathcal{W}_x)$.
	
	\begin{proposition}
	    
	 \label{property_unification}
		Пусть $\phi \in Form(\vv{p})$ и $\sigma: Form(\vv{p}) \to Form(\vv{p})$ --- подстановка. Тогда: 
		\begin{enumerate}[label*= \upshape (\roman*)] 
			\item  Для любой модели Крипке $\mathcal{W}$  выполнено $$(\sigma(\mathcal{W})\vDash \phi  \Longleftrightarrow \mathcal{W} \vDash \sigma(\phi));$$
			\item $\vdash_L \sigma(\phi)$ тогда и только, когда $\sigma(\w) \vDash \phi$ для всех всех моделей Крипке $\w$;
			\item Для любых подстановки $\tau$  и  модели Крипке $\w$ выполнено $$\sigma(\tau(\w)) {=} (\sigma\tau)(\w).$$
		\end{enumerate}
    \end{proposition}
	\begin{proof}
		(i) устанавливается индукцией по построению формулы $\phi$. (ii) следует из (i) и полноты логики $L$. (iii): для любых $x\in \w$ и $p \in \vv{p}$:
		$$\sigma(\tau(\w)), x \Vdash p \overset{def}{\Longleftrightarrow} \tau(\w), x \Vdash \sigma(p) \overset{\text{(i)}}{\Longleftrightarrow}  \w, x \Vdash \sigma(\tau(p)) \overset{def}{\Longleftrightarrow} (\sigma\tau)(\w), x \Vdash p.$$
		
	\end{proof}

	\subsection{Проективность} \label{section_projectivity}
	
	Формула $\phi$ называется \textit{проективной} (в логике $L$), если для неё существует такой унификатор $\sigma: Form(\vv{p}) \to Form(\vv{p})$, что для любого $p\in \vv{p}$ выполнено 
	\begin{equation} \label{projective_condition}
		\phi \vdash_L \sigma(p) \leftrightarrow p \tag{P}.
	\end{equation}
	
	Заметим, что проективный унификатор  для $\phi$ сразу же является самым общим: если $\tau$ другой унификатор для $\phi$, то, из условия (\ref{projective_condition}), имеем $\tau(\phi) \vdash_L \tau(\sigma(p)) \leftrightarrow \tau(p)$ для любой переменной $p$, то есть $\vdash_L \tau(\sigma(p)) \leftrightarrow \tau(p)$, откуда $\tau\leqslant\sigma$, поскольку $\tau(\phi)$ является теоремой $L$.
	
	Ввиду теоремы о подстановке,  условие (\ref{projective_condition})  эквивалентно следующему: 
	$$\phi \vdash_L \sigma(\psi)  \leftrightarrow \psi \text{\quad 	для любой формулы $\psi \in Form(\vv{p})$.}$$
Также заметим, что верно следующие предложение: 
	
	\begin{proposition}\label{projective_composition}
		Множество подстановок, удовлетворяющих свойству (\ref{projective_condition}), замкнуто относительно композиции, независимо от того, унифицируют ли они $\phi$ или нет.
	\end{proposition}  

	\begin{proof}
		Действительно, пусть $\sigma$ и $\tau$ --- две подстановки, удовлетворяющие свойству (\ref{projective_condition}). По наблюдению выше для $\tau$: $\phi \vdash_L \tau(\phi) \leftrightarrow \phi$, следовательно $\phi \vdash_L \tau(\phi)$. Применив $\tau$  к условию (\ref{projective_condition}) для $\sigma$, имеем: $\tau(\phi)  \vdash_L \tau(\sigma(p))\leftrightarrow \tau(p)$, откуда  $\phi \vdash_L \tau(\sigma(p))\leftrightarrow \tau(p)$. Из транзитивности, $\phi \vdash_L \tau(\sigma(p)) \leftrightarrow p$.
		
	\end{proof}

	\section{Основные результаты}

	\subsection{Проективность и свойство расширения для $\j$} 
	
	Мы покажем, что  существование проективного унификатора для формулы $\phi$  в логике $\j$ равносильно наличию некторого свойства для её класса стратифицированных моделей $MOD_S(\phi)$, относительно которого есть полнота, по теореме  \ref{stratification}.
	
	Отныне и далее, слово <<{модель}>>, будем понимать, как \textit{конечная стратифицируемая модель с корнем}.
		\begin{definition}\label{1-congruent}
		Две модели $\w$ и $\w'$  назовём \textit{1-подобными} (обозначение $\w \approx_1 \w'$), если модели (без корня), полученные из них удалением 1-листа, корня совпадают.
	\end{definition}

	\begin{definition}
			\textit{Вариантом} модели Крипке $\mathcal{W}=\l W, R_0, R_1, r, v\r$ называется такая модель Крипке $\mathcal{W}'=\l W,  R_0, R_1, r, v'\r$, что для всех миров $x \in W$ выполнено $$x\ne r \Rightarrow v(x)=v'(x).$$
			
			Класс $K$ моделей Крипке обладает \textit{свойством расширения}, если для любой модели  $\mathcal{W}=\l W, R_0, R_1, r, v\r$, удовлетворяющей условию 
			$$ x\ne r  \Rightarrow \mathcal{W}_x  \in K \text{\quad для всех миров }x \in  W,$$
			найдётся её вариант $\mathcal{W}'$, такой что $\mathcal{W}' \in K$.
	\end{definition}
	
	Таким образом, мы планируем доказать следующую теорему: 
	
	\begin{theorem} \label{extention property}
		Формула $\phi$ проективна в логике $\mathbf{J}_2$  тогда и только тогда, когда класс её стратифицированных моделей $MOD_S(\phi)$ обладает свойством расширения. 
	\end{theorem}

	\begin{proof} $(\Rightarrow)$ Рассмотрим произвольную модель $\mathcal{W}$, такую что для любого мира $x$, кроме корня, выполнено $\mathcal{W}_x \in MOD_S(\phi) $. Поскольку $\phi$ проективна, то ввиду утверждения \ref{property_unification}, $\sigma(\w) \in MOD_S(\phi)$, где $\sigma$ --- соответствующий унификатор. Утверждается, что $\sigma(\w)$ --- искомый вариант для $\w$.  Действительно, так как $\w_x \in MOD_S(\phi)$ для любого мира $x \ne r$,  то $\w, x \Vdash \phi$. Тогда имеем: 
		$$\sigma(\w), x \Vdash p \ \overset{def}{\Longleftrightarrow} \ \w, x \Vdash \sigma(p) \ \overset{(\ref{projective_condition})}{\Longleftrightarrow}  \ \w, x \Vdash p,$$
		что и требовалось показать. 
		
	($\Leftarrow$) Обратную импликацию доказать значительно сложнее. Этому будет посвящён весь оставшийся раздел. 
		
	\end{proof}
	
	Итак, пусть нам дано, что $MOD_S(\phi)$ обладает свойством расширения. Нам необходимо построить для формулы $\phi$ её проективный унификатор. 
	
	Положим $n = d(\phi)$. Сделаем  следующие очень важное наблюдение:
	
	\begin{lemma} \label{lemma1}
		Для любой модели $\w$ логики $\mathbf{J}_2$ с корнем $r$, у которой
		$$\forall x \in \w \ (rR_0x \Rightarrow \w, x \Vdash \phi),$$ 
		найдётся  подстановка $\theta_\w$, удовлетворяющая свойству (\ref{projective_condition}) для формулы $\phi$, такая что: 
			\begin{enumerate}[label=\upshape(\roman*) ]
				\item \label{lemma1_1}$\theta_\w(\w) \vDash \phi$;
				\item \label{lemma1_2}для любой другой модели $\w'$ и $x \in \w'$ выполнено: 
				$\w'_x \vDash \phi \Rightarrow \theta_\w(\w'_x)  =\w'_x$;
				\item \label{lemma1_4} если для некоторой модели $\w'$, найдётся модель $\w''$, такая что $\w' \approx_1 \w''$ и $\w\sim_{n+1} \w''$, то также $\theta_\w(\w') \vDash \phi$. 
			\end{enumerate} 
	\end{lemma}

		\begin{proof}
		\ref{lemma1_1}. Пусть нам дана модель $\w$, в которой формула $\phi$ истинна во всех мирах, кроме некоторых из 1-листа корня. Обозначим этот лист $\mathcal{A}$. 
		
		 Заметим, что поскольку для любого мира из $\mathcal{A}$ по отношению $R_0$ достижимы все остальные 1-листы, то оценка формул вида $[0]\psi$ во всех мирах $\mathcal{A}$  одинакова. Тогда заменим все максимальные подформулы вида $[0]\psi$ в  $\phi$ на их оценку ($\top$  или $\bot$) в листе $\mathcal{A}$  и обозначим полученную формулу $\phi'$. 
		 
		 Так как в формуле $\phi'$ осталась только одна модальность $[1]$, то 1-лист $\mathcal{A}$ можно рассматривать как модель логики $\mathbf{GL}$ по отношению $R_1$. Рассмотрим класс моделей Крипке $MOD_{GL}(\phi')$ логики $\mathbf{GL}$,  в которых истинна формула $\phi'$, и покажем, что он обладает свойством расширения.  В самом деле, пусть для некоторой модели $\mathcal{M}$ выполнено $\mathcal{M}_x \in MOD_{GL}(\phi')$  для любого мира $x$,  кроме корня. Тогда рассмотрим модель $\w^\mathcal{M}$, полученную заменой 1-листа $\mathcal{A}$ из $\w$ на модель $\mathcal{M}$, присоединённую ко всем остальным листам по отношению $R_0$. По построению,  $(\w^\mathcal{M})_x \in MOD_S(\phi)$ для каждого мира $x\ne r$ (в верхних листах построенной модели истинность $\phi$ не изменилась, а в $\mathcal{A}$ истинность $\phi$ и $\phi'$ эквивалентна), значит для модели $\w^\mathcal{M}$ есть вариант $(\w^\mathcal{M} )'\in MOD_S(\phi)$ (по предположению, $MOD_S(\phi)$ обладает свойством расширения), значит 1-лист корня $(\w^\mathcal{M})'$ является вариантом для $\mathcal{M}$. 
		 
		 Следовательно, по теореме Гилярди для логики $\mathbf{GL}$\footnote{\cite{ghilardi2000best}, теорема 2.2.}, у формулы $\phi'$ существует проективный унификатор $\sigma$ (напомним, что в формуле $\phi'$, кроме модальности $[1]$, других нет). Тогда, по  утверждению \ref{property_unification}, $\forall x \in \mathcal{A} \ (\sigma(\mathcal{A}), x \Vdash \phi')$. Подправим теперь $\sigma$  так, чтобы в тех мирах, где была истинна формула $\phi$, истинность не менялась. 
		 
		 Рассмотрим подстановку $\theta_\w$, определяемую как:
		 $$\theta_\w(p_i)=(\phi\land p_i)\vee (\neg \phi \land \sigma(p_i)).$$
		 Свойство (\ref{projective_condition}) для $\theta_\w$ получаем немедленно по построению.
		 
		Для произвольной модели $\w'$, по определению, 
		 $\theta_{\w}(\w'), x \Vdash p_i$ эквивалентно $\w', x \Vdash \theta_\w(p_i)$, поэтому если $\w', x \Vdash \phi$, то  $$\w', x \Vdash \theta_\w(p_i) \Longleftrightarrow \w', x \Vdash p_i.$$ Таким образом,  $\theta_\w(\w')=\w'$, и  утверждение \ref{lemma1_2} леммы доказано. 
		 
		 По предыдущему наблюдению, для любого мира $x\in \w$ не из 1-листа корня, имеем $\theta_\w(\w), x \Vdash \phi$ и $\w \approx_1\theta_\w(\w)$. Теперь пусть $x\in \mathcal{A}$. Если $\w, x \Vdash \phi$, то $\theta_\w$ оставила оценку переменных в мире $x$ прежней, ровно как и $\sigma$ на $\mathcal{A}$ (поскольку $\sigma$ --- проективная: $\phi' \vdash \sigma(p_i)\leftrightarrow p_i$ и $\w, x \Vdash \phi\leftrightarrow \phi'$). Если же $\w, x \not\Vdash \phi$, то $\theta_\w$ действует на мире $x$ как $\sigma$. Иными словами, модель $\sigma(\mathcal{A})$ изоморфна модели $\theta_\w(\mathcal{A})$, которая, в свою очередь, является подмоделью $\theta_\w(\w)$. Тогда:
		 $$ \sigma(\mathcal{A}), x \Vdash \phi'\Longrightarrow \theta_\w(\w), x \Vdash \phi' \Longleftrightarrow \theta_\w(\w), x \Vdash \phi.$$
		 
		 Таким образом, для любого $x \in \w$ верно $\theta_\w(\w), x \Vdash \phi$, и утверждение \ref{lemma1_1} леммы доказано. 
		 
		 \ref{lemma1_4}. Поскольку $\w'' \sim_{n+1} \w$, то, по предложению \ref{bissimulation}, в модели $\w''$ формула $\phi$ истинна во всех мирах, кроме 1-слоя корня. Более того, истинность формул $\phi$ и $\phi'$ в 1-слое корня $\w''$ эквивалентны (поскольку в 1-слоях $\w$ и $\w''$  эквивалентна истинность всех подформул вида $[0]\psi$). Так как $\w'' \approx_1 \w''$, то всё сказанное верно для модели $\w'$. 
		 
		Посмотрим внимательно на 1-слой корня $\mathcal{A}'$ модели $\w'$: поскольку $\sigma$ --- проективный унификатор, то $\sigma(\mathcal{A}') \vDash \phi'$. Но, как было замечено выше, $\sigma(\mathcal{A}') $ и $ \theta_{\w}(\mathcal{A}')$ изоморфны. Отсюда $\theta_{\w}(\mathcal{A}') \vDash \phi'$ и $\theta_{\w}(\w') \vDash \phi$ (истинность $\phi$ и $\phi'$ в $\mathcal{A}'$ совпадает).
		 
	\end{proof}

	Выберем теперь из каждого класса эквивалентности  по отношению $\sim_{n+1}$  представителя $\w$, удовлетворяющему условию леммы \ref{lemma1}, и рассмотрим подстановку $\overline{\theta}$, равную произведению $\theta_{\w}$ из леммы \ref{lemma1} по всем выбранным $\w$. Поскольку классов эквивалентности конечное число, то произведение тоже будет конечным. По предложению \ref{projective_composition}, $\overline{\theta}$ тоже удовлетворяет свойству (\ref{projective_condition}).
	
	\begin{lemma} \label{lemma2}
			Если для модели $\w$ формула $\phi$ истинна всюду, кроме 1-слоя корня, то $\t(\w) \vDash \phi$. 
	\end{lemma} 

	\begin{proof}
		В самом деле, для модели $\w$ найдётся представитель $\w'$, который входит в произведение $\t$. Разложим $\t=\theta_1\theta_{\w'}\theta_2$. Следовательно, по утверждению \ref{lemma1_2} леммы~\ref{lemma1}, $\theta_2(\w)\approx_1 \w$. Тогда, по пункту \ref{lemma1_4} той же леммы, для моделей $\theta_2(\w)$, $\w$ и $\w'$:
		$$\left.\begin{aligned}
							\theta_2(\w) \approx_1 \w\\
							 \w' \sim_{n+1} \w
						\end{aligned}\right\} \Rightarrow \theta_{\w'}(\theta_2(\w)) \vDash \phi.$$
		Сомножитель $\theta_1$ дальше эту истинность сохранит. 
		
	\end{proof}
	Из доказательства выше становится понятно, для чего мы ранее  доказывали пункт \ref{lemma1_4} леммы \ref{lemma1} в такой общности ---  казалось, намного проще было бы взять  $\w''=\w'$ и получить другое интуитивное ясное утверждение: $$\w' \sim_{n+1} \w \Rightarrow \theta_{\w}(\w')\vDash \phi.$$ Однако для унификации модели  подстановкой $\t$ требуются более тонкие соображения,  о которых мы заблаговременно позаботились.
	
	Наша последняя цель --- предъявить подстановку, удовлетворяющую условию (ii) утверждения \ref{property_unification}, и тем самым завершить доказательство теоремы. 
	
	Зафиксируем модель $\w$ и рассмотрим $x\in \w$. Введём необходимые обозначения и определения: 
	\begin{itemize}
		\item $\w[\phi] = \{x\in \w \ | \ \w_x \vDash \phi\}$ --- множество миров модели $\w$, в порождённых подмоделях которых истинна формула $\phi$. 
		\item {\textit{Ранг}}  $rk(x) = \#\{[\w_y]_{n+1} \mid  xR_0y  \with y \! \in \! \w[\phi]\}$ --- мощность множества классов эквивалентности подмоделей  $\w_y$ по всем мирам $y \in \w$, которые $R_0$-достижимы из $x$ и в порождённых подмоделях которых истинна формула~$\phi$.

	\end{itemize}

	Заметим, что, по лемме \ref{lemma1} для моделей $\w$ и $\t(\w)$, $rk(x) \leqslant rk(\t(x))$  для любой мира $x \in \w$. Если в модели $\w$ найдётся мир $x$, такой что $\w, x \not\Vdash \phi$, то найдётся также мир $y$, такой что $\w, y\not\Vdash \phi$ и $\forall z \in \w_y \ (yR_0z \Rightarrow \w_y, z \Vdash \phi$). Тогда, по лемме~\ref{lemma2}, $\t(\w_y)\vDash \phi$, и количество миров $\w$, в которых истинна $\phi$, возросло. Поэтому, применяя  подстановку $\t$ к модели $|\w|$ раз, формула $\phi$ всюду станет истинной.
	 
	  Однако мощность модели не  самый подходящий параметр, поскольку может  неограниченно возрастать. Мы  покажем, что для унификации любой модели достаточно применить подстановку $\t$ всего $N$ раз, где $N$ --- количество классов эквивалентности по отношению~$\sim_{n+1}$. 
	
	\textit{Минимальным рангом} модели назовём число
	$$\mu(\w)=\min\limits_{x\notin \w[\phi]}rk(x).$$
	Мы хотим показать, что если $\w\not\vDash\phi$, то $\mu(\w)<\mu(\t(\w))$. Поскольку минимальный ранг ограничен сверху числом $N$ и дискретно возрастает, то теорема, наконец, будет доказана ($(\t)^N(\w)\vDash \phi$ для любой модели $\w$). 
	
	Итак, предположим, что  $\w\not\vDash \phi$ и $\mu(\w){=}\mu(\t(\w))$. В частности, для таких миров $x\in\w$,  на которых достигался минимум $\mu(\w)$, верно $rk(x)=rk(\t(x))$. Покажем, что в этом случае $\t(\w_x) \vDash \phi$ (какой бы ни был минимальный мир), откуда получим немедленное противоречие с неизменностью минимального ранга. 
	
	Поскольку $\w_x \not\vDash \phi$, то найдётся мир $y\in \w_x$,  такой что по-прежнему $\w_y \not\vDash \phi$ и $(yR_0z \Rightarrow \w_y, z \Vdash \phi)$. По лемме \ref{lemma2}, $\overline{\theta}(\w_y) \vDash \phi$, поэтому разложим $\t=\theta_1\theta_{\w'}\theta_2$, где $\w' \sim_{n+1} \w_y$. Тогда  $\theta_{\w'}(\theta_2(\w_y)) \vDash \phi$. Для краткости положим $\tilde{\w}{=}\theta_2(\w_x)$ и $\theta'=\theta_{\w'}$. По лемме \ref{lemma1} \ref{lemma1_2}, по-прежнему в любом мире $\tilde{\w}_y$ (напомним, что ограничения моделей на миры коммутируют с подстановками), $R_0$-достижимом из $y$, истинна формула $\phi$   и $rk(x)=rk(\theta'(x))$, где мир $x$ теперь берётся из $\tilde{\w}$.

	\begin{lemma}\label{lemma3}
		$\theta'(\tilde{\w}_z)$ является моделью  $\phi$ для любого мира $z\in \tilde{\w}$.
	\end{lemma}

	\begin{proof}
		Если $z\in \tilde{\w}[\phi]$, то утверждение  доказано по лемме \ref{lemma1} \ref{lemma1_2}. Поэтому пусть $z\not\in \tilde{\w}[\phi]$.

			Заметим, что  из свойств стратифицированных моделей, выполнено включение
			$$\{[\tilde{\w}_v]_{n+1} \mid  zR_0v \with v \! \in \! \tilde{\w}_z[\phi]\} \subset \{[\tilde{\w}_v]_{n+1} \mid  xR_0v \with v\! \in \! \tilde{\w}_x[\phi]\}.$$

			В то же время, из леммы \ref{lemma1} \ref{lemma1_2}, верно включение 
			$$\{[\tilde{\w}_v]_{n+1} \mid  xR_0v \with v\! \in \! \tilde{\w}_x[\phi]\} \subset \{[\theta'(\tilde{\w}_v)]_{n+1} \mid  xR_0v \with v\! \in \! \theta'(\tilde{\w}_x)[\phi]\}.$$
			
			 Из минимальности ранга: $rk(y)=rk(z)=rk(x)$ и, по предположению, все эти ранги равны также $rk(\theta'(x))$. Значит, все  три множества выше совпадают, поскольку они конечны и  имеют одинаковое количество элементов. 
			 
			 Докажем утверждение леммы индукцией по $h_0(z)$ при, где  $h_i(z)$ --- длина наибольшей цепи  $x_1R_ix_2R_i \ldots R_ix_m$, где $x_1=z$ и 
			 $x_k\not\in \tilde{\w}[\phi]$ при $k=1, \ldots, m$.

			 \textit{База:} $h_0(z)=0$. Тогда формула $\phi$ истинна во всех мирах  модели $\tilde{\w}_z$, а значит, по лемме \ref{lemma1} \ref{lemma1_2}, подстановка $\theta'$ эту истинность сохранит.
			 
			 \textit{Переход:} по предположению индукции, формула $\phi$ истинна во всех мирах модели $\theta'(\tilde{\w}_z)$, кроме 1-слоя корня (обозначим его  $\mathcal{A}$). Нам осталось показать, что $\phi$ также истинна в любом мире из $\mathcal{A}$.

			 Пусть $\sigma$ --- проективный унификатор формулы $\phi'$, полученной заменой оценок всех подформул $[0]\psi$ на их оценку в 1-слое корня модели $\w'$. 
			 Поскольку $\phi$ истинна вне 1-слоя корня $\theta'(\tilde{\w}_z)$ и множества классов эквивалентностей равны, то для любого $a\in \mathcal{A}$ имеем $\theta'(\tilde{\w}_z), a \Vdash \phi \leftrightarrow \phi'$ (истинность подформул вида $[0]\psi$ в 1-слое  с моделью $\w'$ эквивалентны).

			Для любого мира  $a\in \mathcal{A}$ индукцией по $h_1(a)$ докажем, что $\theta'(\a_a)=\sigma(\a_a)$. Если $h_1(a)=0$, то $\mathcal{A}_a$ является моделью формулы $\phi$, значит, по лемме \ref{lemma1} \ref{lemma1_2},  $\theta'(\a_a) = \a_a$. Тогда   $\theta'(\a_a) \vDash \phi'$ и  $\a_a \vDash \phi'$, поскольку $\theta'(\a_a)\vDash \phi \leftrightarrow \phi'$. Следовательно, модели $\theta'(\tilde{\w}_a)$ и  $\sigma(\tilde{\w}_a)$ совпадают:  $\sigma$ --- проективный унификатор ($\phi' \vdash \sigma(p)\leftrightarrow p$) и обе подстановки  модель не изменили. 

			Пусть для всех миров $t$, таких что $aR_1t$, выполнено $\theta'(\tilde{\w}_t)=\sigma(\tilde{\w}_t)$. Если в мире $a$ истина формула $\phi$, то рассуждения аналогичны базе. Если нет, то:
				\begin{equation*}
				\theta'(\tilde{\w}_a), a \Vdash p \overset{def}{\Longleftrightarrow} \tilde{\w}_a,a \Vdash \theta'(p)\overset{a\not\Vdash\phi}{\Longleftrightarrow}\tilde{\w}_a, a \Vdash \sigma(p), 
			\end{equation*} 
		то есть снова $\theta'(\tilde{\w}_a)=\sigma(\tilde{\w}_a)$.
			
			Итак, для 1-слоя корня  $\mathcal{A}$ выполнено $\theta'(\mathcal{A})=\sigma(\mathcal{A})$. 	Но, по утверждению \ref{property_unification}, $\theta'(\tilde{\w}_z), a \Vdash \phi'$ для всех $a\in \mathcal{A}$,  а значит  $\theta'(\tilde{\w}_z) \vDash \phi$, поскольку истинность $\phi$ и $\phi'$ в 1-слое модели $\theta'(\tilde{\w}_z)$ одинакова.

	\end{proof}

	В только что доказанной лемме в качестве $z$ возьмём мир $x$: $\theta'(\tilde{\w}_x) \vDash \phi$, и сомножитель $\theta_1$ далее эту истинность сохранит. Значит, в модели $\t(\w_x)$ формула $\phi$  является истинной  --- противоречие с выбором $\w_x$!
	
	Таким образом, $\theta =(\t)^N$ --- искомый проективный унификатор для формулы $\phi$, и теорема \ref{extention property}, тем самым, доказана.

	\subsection{Финитный тип унификации $\j$} 
	
	Теперь мы хотим показать, что любая унифицируемая формула в логике $\mathbf{J}_2$, имеет конечный базис унификаторов. 
	
	 Мы докажем, что \textit{для любого унификатора $\sigma$ для $\phi$ найдётся некоторая проективная формула $\psi$ глубины не более $d(\phi)$, для которой $\sigma$ также является унификатором и $\psi \vdash_{\mathbf{J}_2} \phi$}. Тогда $\sigma$ будет менее общий, как подстановка,  самого общего унификатора $\psi$, который в свою очередь также является унификатором для $\phi$ по последнему условию. Конечность базиса будет следовать из существования  конечного до доказуемой эквивалентности множества формул глубины не более $d(\phi)$. 
	
	\begin{definition}
		Для класса ${K} \subseteq  MOD_S$ стратифицированных моделей  и некоторой модели $\w$ будем писать $\w \leqslant_n \mathcal{K}$, если для всех $w \in \w$ найдётся модель $\mathcal{U} \in  K$ и мир $u \in \mathcal{U}$, такие что $\w_w \sim_{n} \mathcal{U}_u$.
	\end{definition}
	
	\begin{proposition}
		Класс $K \subseteq MOD_S$ стратифицированных моделей имеет вид $MOD_S(\phi)$ для некоторой формулы $\phi$ глубины не более $n$ в том и только в том случае, если $K$ удовлетворяет следующему условию: 
		$$\w \leqslant_n K \Rightarrow \w \in K \quad \text{для всех  $\w \in MOD_S$.}$$
	\end{proposition}
	
	\begin{proof}
		 В одну сторону утверждение следует непосредственно из предложения \ref{bissimulation} про $n$-бисимуляцию. В обратную --- напомним, что, из доказательства предложения \ref{bissimulation}, для каждой модели $\w$ существует формула $X^n_\w$, такая что 
		 		\begin{flalign*}\tag{$*$} \label{proposition_closure}
		 	\mathcal{W}', x' \Vdash X^n_{\mathcal{W}} \ \Longleftrightarrow \ \mathcal{W} \sim_n \mathcal{W}' \quad \text{для всех моделей $\w'$}. 
		 \end{flalign*}
Тогда в качестве $\phi$ возьмём формулу $\bigvee_{\w \in K} \bigvee_{w \in \w} X^n_{\w_w}$ (дизъюнкция, на самом деле, будет конечной, поскольку существует только конечное до эквивалентности число формул вида $X^n_\w$). По построению, $K \subseteq MOD_S(\phi)$. Обратно, если в некоторой модели $\w$ истинна формула $\phi$, то в ней истинен один из дизъюнктов, а значит, по условию (\ref{proposition_closure}) и замкнутости $K$ относительно $\leqslant_n$, $MOD_S(\phi) \subseteq K$. 
	 
	\end{proof}

	\begin{definition}
			Класс $K$ стратифицированных моделей называется \textit{стабильным}, если из $\w \in K$ влечёт $\w_x \in K$ для всех миров $x \in \w$. 		
	\end{definition}

В частности, класс $MOD_S(\phi)$ является стабильным. 

\begin{proposition} \label{proposition_stability}
	Пусть $K$  ---  стабильный класс стратифицированных моделей. Тогда $$K_n =\{\w \in MOD_S \ |  \ \forall w \in \w \  \exists \: \mathcal{U} \in K \ \mathcal{U} \sim_{n} \w_x\}$$
	является наименьшим классом, расширяющим $K$, вида $MOD_S(\phi)$ для некоторой формулы $\phi$ глубины не более $n$. 
\end{proposition}

\begin{proof}
	Заметим, что $\w \in K_n \Longleftrightarrow \w \leqslant_n K$.
	Проверим, что для $K_n$ выполнено условие предыдущего предложения: $$\w \leqslant_n K_n \Longrightarrow \w \leqslant_n K \Longrightarrow \w \in K_n.$$
\end{proof}
\begin{definition}
	Пусть $\{(\w_i, r_i)\}^n_{i=1}$ --- семейство попарно 1-подобных (см. определение \ref{1-congruent}) моделей c корнем. Тогда  их \textit{$1$-суммой} $\prescript{1}{}{\sum_{i=1}^n} \w_i $ называется  модель $(\w, r)$, такая что:
	\begin{itemize}
		\item $(\w, r) \approx_1 (\w_i, r_i)$ для каждого $i$;
		\item нижний 1-лист корня $\w$ получается из 1-листов $\w_i$ присоединением нового корня $r$ по отношению $R_1$ с пустой оценкой переменных.
	\end{itemize}
\end{definition}
\begin{proposition} \label{n-sum}
	Пусть $K$ --- класс стратифицированных моделей логики $\mathbf{J}_2$, такой что для любых двух моделей $\w$ и $\w'$
	\begin{itemize}
		\item $\w \sim_\infty \w', \w \in K \Rightarrow \w' \in K$;
		\item $\w \in K \Rightarrow \w_x \in K$ для любого $x$ (то есть класс $K$ является стабильным).
	\end{itemize}
	Тогда $K$ обладает свойством расширения тогда и только тогда, когда для любого конечного множества попарно 1-подобных моделей, лежащих в $K$, найдётся вариант их 1-суммы из класса $K$. 
\end{proposition}

\begin{proof}
	Если $K$ обладает свойством расширения, то для любого конечно множества попарно 1-подобных моделей их 1-сумма $\prescript{1}{}{\sum_{i=1}^n} \w_i $ обладает условию расширения. Значит, для неё найдётся вариант из $K$. 
	
	Обратно, пусть у некоторой модели $(\w, r)$ для каждого мира, кроме корня, выполнено $\w_x\in K$. Для всех миров  $y$, $R_1$-достижимых из $r$, рассмотрим 1-сумму  $\tilde{\w} := \prescript{1}{}{\sum \w_y}$ попарно 1-подобных моделей. Тогда в $K$ найдётся вариант $\tilde{\w}'$ для $\tilde{\w}$. Несложно проверить, что $\w \sim_\infty \tilde{\w}$. Пусть $\w'$ --- вариант для $\w$ с оценкой переменных в корне, как у модели $\tilde{\w}'$. В этом случае, $\tilde{\w}' \sim_\infty \w'$, а значит $\w' \in K$.
	
\end{proof}

\begin{theorem} \label{theorem_basis1}
	Для любой унифицируемой формулы $\phi$ в логике $\mathbf{J}_2$ существует конечный базис унификаторов.
\end{theorem}

\begin{proof}
	Пусть $\sigma$ -- какой-нибудь унификатор для $\phi$, и положим $n=d(\phi)$. Мы хотим показать, что существует такая проективная формула $\psi$, что $\psi \vdash_{\mathbf{J}_2} \phi,  d(\psi)\leqslant n$ и $\sigma$ является унификатором для $\psi$. Тогда $\sigma$ будет менее общий самого общего унификатора для $\psi$, который, в свою очередь, будет являться унификатором для формулы $\phi$. 
	
	Рассмотрим класс $$K = \{\w \in MOD_S \ | \ \exists  \u\in MOD_S \ \w \sim_\infty \sigma(\u)\}$$
	и его расширение $K_n$ (см. предложение  \ref{proposition_stability}). Непосредственно из определения, $K_n$ удовлетворяет условиям предложения \ref{n-sum} (стабильность очевидна; из $\sim_\infty$ безусловно следует $\sim_{n}$). Более того, мы утверждаем, что класс $K_n$ обладает свойством расширения. 
	
  Пусть дано семейство попарно 1-подобных моделей $\{(\w_i, r_i)\}^k_{i=1}$  из $K_n$. Это значит, что для каждой модели $\w_i$ найдётся модель $\u_i$, что $\w_i \sim_n \sigma(\u_i)$. Поскольку $\u_i$ не обязаны быть попарно 1-подобными, то с помощью модели $\u_1$ мы переделаем остальные $\u_i$: для каждого $i>1$ рассмотрим новую модель $\tilde{\u}_i$, такую что нижний 1-лист $\tilde{\u}_i$ совпадает с 1-листом корня $\u_i$ и $\tilde{\u}_i \approx_1\u_1$ (по определению, положим $\tilde{\u}_1:=\u_1$). Теперь $\tilde{\u}_i$ попарно 1-подобны и, так как $\w_i$  тоже попарно 1-подобны, то $\w_i \sim_n\sigma(\tilde{\u}_i)$ для каждого $i$. 
  
  Рассмотрим 1-сумму моделей $\tilde{\u}:= \prescript{1}{}{\sum_{i=1}^n} \tilde{\u}_i$ и возьмём в качестве  варианта для $\w :=\prescript{1}{}{\sum_{i=1}^n} \w_i$ модель $\w'$, оценка  переменных в корне которой совпадает с оценкой переменных в корне $\sigma(\tilde{\u})$.  Несложно проверить, что $\w' \sim_n \sigma(\tilde{\u})$. Тогда, $\w'\in K_n$, и значит, по предложению \ref{n-sum}, класс $K_n$ обладает свойством расширения. 
  
  Теперь, применяя предложение \ref{proposition_stability} к классу $K_n$, имеем $K_n=MOD_S(\psi)$ для некоторой формулы $\psi$, глубины не более $n$. Как мы показали выше, $MOD_S(\psi)$ обладает свойством расширения, поэтому по теореме \ref{extention property}, формула $\psi$ является  проективной. Более того, $\sigma$ является унификаторм для $\psi$, поскольку для любой модели $\w$,  $\sigma(\w)\in K \subseteq MOD_S(\psi)$ (см. утверждение \ref{property_unification}). Осталось только убедиться, что $\psi \vdash_{\mathbf{J}_2} \phi$.

	Поскольку $\sigma$ --- унификатор для $\phi$, то $\sigma(\w) \vDash \phi$ для любой модели $\w \in MOD_S$. 	Значит, $K \subseteq MOD_S(\phi)$, поскольку класс $MOD_S(\phi)$ замкнут относительно $\sim_\infty$. Так как $d(\phi)\leqslant n$, то $K_n \subset MOD_S(\phi)$, по предложению \ref{bissimulation}. Таким образом, 
	$$K_n = MOD_S(\psi)\subseteq MOD_S(\phi),$$ 
	следовательно, $\psi \vdash_{\mathbf{J}_2}\phi$ ввиду полноты логики $\mathbf{J}_2$.
	
\end{proof}

	Теперь мы можем переформулировать доказательство теоремы \ref{theorem_basis1}, используя концепт \textit{проективной аппроксимации}  формулы $\phi$. Обозначим через $L$ логику $\mathbf{J}_2$ или логику \textbf{GLB}.
	
	 Пусть $S(\phi)$ --- множество проективных формул $\psi$, таких что $d(\psi) \leqslant d(\phi)$ и  $\psi \vdash_L \phi$. 
	 
	 \begin{definition}
	 	\textit{Проективной аппроксимацией} $\Pi(\phi)$ формулы $\phi$ называется минимальное подмножество $S(\phi)$, такое что для любой формулы $\psi \in S(\phi)$ найдётся формула $\gamma \in \Pi(\phi)$, такая что $\psi \vdash_L  \gamma$. 
	 \end{definition}
	
	Иными словами, проективная аппроксимация $\phi$ получается из $S(\phi)$ оставлением одной формулы из $\vdash_L$-максимального класса. Теорема \ref{theorem_basis1} утверждает, что каждый унификатор для $\phi$ является унификатором для некоторой формулы из $S(\phi)$, а значит и для некоторой формулы из $\Pi(\phi)$. Следовательно, любой унифиактор для формулы $\phi$ является менее общим самого общего унификатора для некоторой формулы из $\Pi(\phi)$. Более того, верно
	
	\begin{proposition}
		Самые общие унификаторы для формул из $\Pi(\phi)$ образуют базис унификаторов для формулы $\phi$ в логике $L=\j$. 
	\end{proposition}

	\begin{proof}
		Ввиду сказанного выше, осталось проверить, что самые общие унификаторы для формул из $\Pi(\phi)$ попарно не сравнимы относительно $\leqslant$. 
		
		Пусть $\psi_1, \psi_2 \in \Pi(\phi)$ --- две проективные формулы и $\sigma_1, \sigma_2$ их самые общие унификаторы соотвественно. Мы утверждаем, что 
		$\psi_1 \vdash_{\j} \psi_2 \Longleftrightarrow \sigma_1 \leqslant \sigma_2.$
		
		В одну сторону, если $\psi_1 \vdash_{\j} \psi_2$, то из $\psi_2 \vdash_{\j} p \leftrightarrow\sigma_2(p)$, имеем $$\sigma_1(\psi_2)\vdash_{\j} \sigma_1(p) \leftrightarrow \sigma_1(\sigma_2(p)),$$ а значит $ \sigma_1 \leqslant \sigma_2$, поскольку $\sigma(\psi_2)$ является теоремой ${\j}$. 
		
		В другую --- если  $\sigma_1 \leqslant \sigma_2$, то найдётся такая подстановка $\tau$, что  $\sigma_1$ эквивалентна композиции $\tau\sigma_2$. В частности, для этой эквивалентности имеем 
		$$\vdash_{\j}  \sigma_1(\psi_2) \leftrightarrow \tau(\sigma_2(\psi_2)).$$
		Поскольку $\vdash_{\j}  \sigma_2(\psi_2)$, то $\tau(\sigma_2(\psi_2))$ также является теоремой ${\j}$, а значит $\vdash_{\j}\sigma_1(\psi_2)$. Но в то же время, $\psi_1\vdash_{\j}\psi_2 \leftrightarrow \sigma_1(\psi_2)$, так что $\psi_1 \vdash_{\j} \psi_2$.
		
	\end{proof}

	Описанный выше алгоритм показывает разрешимость построения конечной проективной аппроксимации для заданной формулы логики $\j$.
		
		\subsection{Описание допустимых правил в $\j$}
		В заключение, мы готовы дать описание допустимых правил  в логике  $\mathbf{J}_2$. Напомним, что правило $\phi_1/\phi_2$ в логике $L$ называется \textit{допустимым}, если для каждой подстановки $\sigma$, такой что $\vdash_L \sigma(\phi_1)$, также $\vdash_L \sigma(\phi_2)$. 
		
		\begin{theorem}
			Правило $\phi_1/\phi_2$  является допустимым в логике $\j$ тогда и только тогда, когда для всех формул $\psi \in \Pi(\phi_1)$ выполнено $\psi\vdash_{\j} \phi_2$. 
		\end{theorem}
	
	\begin{proof}
		Если правило $\phi_1/\phi_2$ является допустимым в ${\j}$ и $\psi \in \Pi(\phi_1)$, то формула $\psi$ является проективной и $\psi \vdash_{\j} \phi_1$. Возьмём произвольный унификатор $\sigma$ для $\psi$: он удовлетворяет свойству $\psi \vdash_{\j} \phi_2 \leftrightarrow\sigma(\phi_2)$ (см. раздел \ref{section_projectivity}). Тогда  $\vdash_{\j}\sigma(\phi_1)$, откуда $\vdash_{\j}\sigma(\phi_2)$. Значит $\psi \vdash_{\j} \phi_2$. 
		
		Наооборот, пусть для всех  $\psi \in \Pi(\phi_1)$ выполнено $\psi\vdash_{\j} \phi_2$ и $\sigma$ --- некоторый унификатор для $\phi_1$. Тогда в $\Pi(\phi)$ должна найтись формула $\psi$, что $\vdash_{\j} \sigma(\psi)$. По предположению, $\psi \vdash_{\j} \phi_2$, следовательно $\vdash_{\j} \sigma(\phi_2)$, то есть правило $\phi_1/\phi_2$  является допустимым.
		
	\end{proof}

	Таким образом, проблема допустимости правил вывода логики $\j$ является разрешимой, поскольку разрешимы построение конечной проективной аппроксимации и проблема выводимости в логике $\j$.

\bibliographystyle{plain}
\bibliography{main.bib}

\end{document}

%% file: macros.tex
\newtheorem{theorem}{Теорема}
\newtheorem{lemma}{Лемма}[section]

\newtheorem{corollary}[lemma]{Следствие}
\newtheorem{proposition}[lemma]{Предложение}

\newtheorem{definition}{Определение}
\newtheorem{example}[lemma]{Пример}
\newtheorem{remark}[lemma]{Замечание}

\newcommand{\bl}{\begin{lemma}}
\newcommand{\el}{\end{lemma}}
\newcommand{\bt}{\begin{theorem}}
\newcommand{\et}{\end{theorem}}
\newcommand{\bcor}{\begin{corollary}}
\newcommand{\ecor}{\end{corollary}}
\newcommand{\bp}{\begin{proof} \ }
\newcommand{\ep}{\end{proof} \medskip}
\newcommand{\bpr}{\begin{proposition}}
\newcommand{\epr}{\end{proposition}}
\newcommand{\brem}{\begin{remark} \em}
\newcommand{\erem}{\end{remark}}
\newcommand{\bd}{\begin{definition} \em}
\newcommand{\ed}{\end{definition}}
\newcommand{\bex}{\begin{example} \em
}
\newcommand{\eex}{\end{example}}
\newcommand{\beq}{\begin{equation} }
\newcommand{\eeq}{\end{equation}}

\newcommand{\bi}{\begin{itemize}
  }
\newcommand{\ei}{\end{itemize}}
\newcommand{\ben}{\begin{enumerate} }
\newcommand{\een}{\end{enumerate} }

\newenvironment{enumr}{

\begin{enumerate}     }{\end{enumerate}

}

\newcommand{\benr}{\begin{enumr}
  }
\newcommand{\eenr}{
\end{enumr}}

\newcommand{\ignore}[1]{}

\newlength{\hilflh}

\renewcommand{\phi}{\varphi}